\documentclass[a4paper]{amsart}

\usepackage{amssymb,bm,amsfonts, stmaryrd, amsmath, amsthm, enumerate, mathtools,shuffle,colonequals}

\usepackage{hyperref}

\usepackage{bibentry}

\usepackage{tikz-cd}

\newcommand{\nospacepunct}[1]{\makebox[0pt][l]{\,#1}}

\usepackage[capitalise]{cleveref}

\crefformat{equation}{(#2#1#3)}
\crefrangeformat{equation}{(#3#1#4--#5#2#6)}
\crefmultiformat{equation}{(#2#1#3}{ \& #2#1#3)}{, #2#1#3}{ \& #2#1#3)}

\crefformat{enumi}{(#2#1#3)}
\crefrangeformat{enumi}{(#3#1#4--#5#2#6)}
\crefmultiformat{enumi}{(#2#1#3}{ \& #2#1#3)}{, #2#1#3}{ \& #2#1#3)}

\theoremstyle{plain}
\newtheorem{theorem}[subsection]{Theorem}
\newtheorem{lemma}[subsection]{Lemma}
\newtheorem{proposition}[subsection]{Proposition}
\newtheorem{corollary}[subsection]{Corollary}

\theoremstyle{definition}
\newtheorem{definition}[subsection]{Definition}
\newtheorem{example}[subsection]{Example}

\theoremstyle{remark}
\newtheorem*{ack}{Acknowledgements}
\newtheorem{remark}[subsection]{Remark}

\numberwithin{equation}{subsection}

\renewcommand{\mod}[1]{\operatorname{mod}(#1)}

\newcommand{\proj}[1]{\operatorname{proj}(#1)}

\newcommand{\ccat}[2][]{\operatorname{Ch}^{#1}(#2)}
\newcommand{\cbcat}[1]{\ccat[b]{#1}}
\newcommand{\clcat}[1]{\ccat[+]{#1}}
\newcommand{\crcat}[1]{\ccat[-]{#1}}
\newcommand{\kcat}[2][]{\operatorname{K}^{#1}(#2)}
\newcommand{\kbcat}[1]{\kcat[b]{#1}}
\newcommand{\dcat}[2][]{\operatorname{D}^{#1}(#2)}
\newcommand{\dbcat}[1]{\dcat[b]{#1}}
\newcommand{\Ac}[2][]{\operatorname{Ac}^{#1}(#2)}

\newcommand{\cat}[1]{{\mathcal{#1}}}
\DeclareMathOperator{\cone}{cone}

\DeclareMathOperator{\Ext}{Ext}
\DeclareMathOperator{\Hom}{Hom}
\DeclareMathOperator{\id}{id}
\DeclareMathOperator{\im}{im}

\newcommand{\stab}[1]{\underline{#1}}
\newcommand{\susp}{\Sigma}
\DeclareMathOperator{\thick}{thick}

\newcommand{\BZ}{\mathbb{Z}}

\newcommand{\sB}{\mathsf{B}}
\newcommand{\sF}{\mathsf{F}}
\newcommand{\sR}{\mathsf{R}}

\title[Realization functors in algebraic triangulated categories]{Realization functors in algebraic triangulated categories}

\date{\today}



\author[J.~C.~Letz]{Janina C. Letz}
\address{Janina~C.~Letz,
Faculty of Mathematics,
Bielefeld University,
PO Box 100 131,
33501 Bielefeld,
Germany}
\email{jletz@math.uni-bielefeld.de}

\author[J.~Sauter]{Julia Sauter}
\address{Julia Sauter,
Faculty of Mathematics,
Bielefeld University,
PO Box 100 131,
33501 Bielefeld,
Germany}
\email{jsauter@math.uni-bielefeld.de}

\begin{document}

\begin{abstract}
Let $\cat{T}$ be an algebraic triangulated category and $\cat{C}$ an extension-closed subcategory with $\Hom (\cat{C}, \susp^{<0} \cat{C})=0$. Then $\cat{C}$ has an exact structure induces from exact triangles in $\cat{T}$. Keller and Vossieck say that there exists a triangle functor $\dbcat{\cat{C}} \to \cat{T}$ extending the inclusion $\cat{C} \subseteq \cat{T}$. We provide the missing details for a complete proof. 
\end{abstract}

\maketitle

\section{Introduction}

Let $\cat{T}$ be a triangulated category and $\cat{C}$ full additive subcategory with an exact structure. A \emph{realization functor} for $\cat{C}$ is a triangle functor $\dbcat{\cat{C}} \to \cat{T}$ extending the inclusion. There are various constructions of a realization functor, all requiring an enhancement and restricting to certain subcategories $\cat{C}$. The first realization functor was constructed in \cite{Beilinson/Bernstein/Deligne:1982} when $\cat{C}$ is the heart of a t-structure in a filtered triangulated category; also see \cite[Appendix]{Psaroudakis/Vitoria:2018}. A different construction appears in \cite{Neeman:1991}.

In this paper we work in algebraic triangulated categories; these include all stable module categories and derived categories. Unlike the works mentioned above we consider exact subcategories of $\cat{T}$, not hearts of t-structures. There exist exact categories whose bounded derived category does not admit a bounded t-structure; see \cite{Neeman:2024}. 

The following result appears in \cite[3.2 Th\'eor\`eme]{Keller/Vossieck:1987}:

\begin{theorem} \label{main}
Let $\cat{T}$ be an algebraic triangulated category and $\cat{C}$ an extension-closed full subcategory with $\Hom_{\cat{T}}(\cat{C},\susp^{-n} \cat{C}) = 0$ for $n \geq 1$. Then $\cat{C}$ has an exact structure induced from the triangulated structure on $\cat{T}$ and there exists a realization functor. 
\end{theorem}

The article \cite[3.2 Th\'eor\`eme]{Keller/Vossieck:1987} provides a sketch of the proof, referring to a construction later appearing in \cite{Keller:1990}. The main goal of this paper is to provide the missing details for a complete proof of \cref{main}. 

The non-negativity condition in \cref{main} for $\cat{C}$ is necessary for our construction. It also appears when the realization functor is a triangle equivalences. In fact, whenever the realization functor is fully faithful, then $\cat{C}$ has to satisfy the non-negativity condition. 

\Cref{main} can be considered the standard tool to realize an (algebraic) triangulated as a bounded derived category of an exact category; we provide conditions for when the realization functor is an equivalence in \cref{equivalence}. Therefore, \cref{main} is expected to be used in classifications of exact subcategories of a triangulated category up to (bounded) derived equivalence. 

Further, finding a realization functor is an alternative to tilting theory. Tilting subcategories in a triangulated category were defined by Keller; see for example \cite{Keller/Krause:2020}. A subcategory $\cat{C}$ of $\cat{T}$ is \emph{tilting}, if $\cat{C}$ is endowed with the split exact structure, hence $\dbcat{\cat{C}} = \kbcat{\cat{C}}$, and the realization functor $\kbcat{\cat{C}} \to \cat{T}$ exists and is a triangle equivalence. There exist realization functors that are equivalences that are not induced by tilting theory; for example the inclusion of a small exact category into its weak idempotent completion induces a triangle equivalence on their bounded derived categories; see \cite[1.10]{Neeman:1990}. 

In general it is not known whether a realization functor of a category $\cat{C}$ is unique. However, it is unique with respect to the chosen enhancement. Theorem \ref{main} is also central to the search for a universal property defining the bounded derived category of an exact category; cf.\@ \cite{Keller:1991} and for derivators \cite{Porta:2015}.

\begin{ack}
We would like to thank the referee for pointing out \cref{counterexample_nonnegative}.

The authors were partly funded by the Deutsche Forschungsgemeinschaft (DFG, German Research Foundation)---Project-ID 491392403---TRR 358. The first author was also partly funded by the Alexander von Humboldt Foundation in the framework of a Feodor Lynen research fellowship endowed by the German Federal Ministry of Education and Research. 
\end{ack}

\section{Realization functor} \label{sec:construction}

The bounded derived category of an exact category $\cat{C}$ is the Verdier quotient of the homotopy category of the underlying additive category by the full subcategory of bounded $\cat{C}$-acyclic complexes $\Ac[b]{\cat{C}}$; see \cite{Neeman:1990} and also \cite[Section~10]{Buehler:2010}. We fix a triangulated category $\cat{T}$ with suspension functor $\susp$. A \emph{realization functor} for an additive subcategory $\cat{C}$ of $\cat{T}$ with an exact structure is a triangle functor $\dbcat{\cat{C}} \to \cat{T}$ extending the inclusion $\cat{C} \to \cat{T}$. 

\subsection{Admissible exact subcategories} 

In this work we focus on subcategories $\cat{C}$ of the triangulated category $\cat{T}$ that inherit their exact structure from the triangulated structure of $\cat{T}$. 

\begin{definition}
A full subcategory $\cat{C}$ is called \emph{non-negative} if $\Hom_{\cat{T}}(\cat{C},\susp^{<0} \cat{C}) = 0$; this means $\Hom_{\cat{T}}(X,\susp^{n} Y) = 0$ for any $X,Y \in \cat{C}$ and $n < 0$. When $\cat{C}$ is additionally closed under extensions and direct summands, we say $\cat{C}$ is \emph{admissible exact}. 
\end{definition}

By \cite{Dyer:extrinotes}, any extension-closed, non-negative subcategory $\cat{C}$ of a triangulated category $\cat{T}$ inherits an exact structure from the triangulated structure: The short exact sequences $L \xrightarrow{f} M \xrightarrow{g} N$ in $\cat{C}$ are precisely those that fit into an exact triangle $L \xrightarrow{f} M \xrightarrow{g} N \xrightarrow{h} \susp L$. 

\begin{remark}
With the notation of `admissible exact' we follow \cite[D\'efi\-ni\-tion~1.2.5]{Beilinson/Bernstein/Deligne:1982} and \cite[Section~2]{Hubery:2016}; the former only considers `admissible abelian', while the latter dropped `exact'. We use admissible exact to avoid confusion with the notions of left/right admissible in the sense of \cite[\S1]{Bondal/Kapranov:1989}. 
\end{remark}

The crucial condition of admissible exactness is the non-negativity. In fact, when $\cat{C}$ is non-negative, then the smallest full subcategory closed under extensions and direct summands containing $\cat{C}$ is an admissible exact subcategory.

\begin{example}
We equip an extension-closed subcategory $\cat{C}$ of an exact category $\cat{E}$ with the induced exact structure; that is $\cat{C}$ is a \emph{fully exact} subcategory of $\cat{E}$. Then $\cat{C}$ is an admissible exact subcategory of $\dbcat{\cat{E}}$. 
\end{example}

\begin{example}
The heart of any t-structure on a triangulated category is admissible exact. Any intersection of admissible exact subcategories is admissible exact. Hence the intersection of two hearts is admissible exact; this applies in particular for hearts that are mutations of each other; see \cite{Chen:2010} for HRS tilting and \cite{Broomhead/Pauksztello/Ploog/Woolf:2023} for the heart fan of an abelian category.
\end{example}

\subsection{Weak realization functor}

Next, we consider triangle functors $\kbcat{\cat{C}} \to \cat{T}$ extending the inclusion for any full subcategory $\cat{C}$ of $\cat{T}$; such a functor can be considered as a realization functor for $\cat{C}$ with the split exact structure. We call such a functor a \emph{weak realization functor}. Under reasonable conditions on the exact structure a weak realization functor induces a realization functor.

\begin{lemma} \label{weak_real_to_real}
Let $\cat{C} \subseteq \cat{T}$ be a full subcategory with an exact structure. We assume there exists a weak realization functor $\sF \colon \kbcat{\cat{C}} \to \cat{T}$. If any exact sequence $L \to M \to N$ in $\cat{C}$ fits into an exact triangle $L \to M \to N \to \susp L$ in $\cat{T}$, then $\sF$ induces a realization functor such that the following diagram commutes
\begin{equation*}
\begin{tikzcd}
\kbcat{\cat{C}} \ar[r] \ar[dr] & \dbcat{\cat{C}} \ar[d] \\
& \cat{T} \nospacepunct{.}
\end{tikzcd}
\end{equation*}
In particular, this holds when $\cat{C}$ is an admissible exact subcategory. 
\end{lemma}
\begin{proof}
It is enough to show that $\sF$ sends acyclic complexes to zero. For this it is enough to show that complexes of the form
\begin{equation*}
(\cdots \to 0 \to L \to M \to N \to 0 \to \cdots) = \cone(\cone(L \to M) \to N)
\end{equation*}
are send to zero when $L \to M \to N$ is an exact sequence in $\cat{C}$. But this holds by assumption.
\end{proof}

\begin{remark}
The above condition on $\sF$, that any exact sequence in $\cat{C}$ fits into an exact triangle in $\cat{T}$, means that $\cat{C} \to \cat{T}$ is a \emph{$\delta$-functor} as defined in \cite{Keller:1991}. 
\end{remark}

In the sequel we construct a weak realization functor. However, we do not know of a general criterion for the existence of a weak realization functor. Our construction requires some form of non-negativity. In particular, a weak realization functor may even exist for $\cat{C}=\cat{T}$.

\begin{example}\label{example:trivial}
Let $k$ be a field and $\cat{T} = \operatorname{vect}(k)$, the category of finite-dimensional $k$-vector spaces with suspension $\susp = \id$. We can view $\kbcat{\cat{T}}$ as the category of finite-dimensional $\BZ$-graded $k$-vector spaces $\operatorname{vect}^\BZ(k)$ with suspension the shift of the grading. The forgetful functor from graded $k$-vector spaces to ungraded $k$-vector spaces is a weak realization functor for $\cat{C}=\cat{T}$. As $\dbcat{\cat{T}} = \kbcat{\cat{T}}$ we obtain the realization functor
\begin{equation*}
\dbcat{\cat{T}} = \operatorname{vect}^\BZ(k) \to \operatorname{vect}(k) = \cat{T}\,.
\end{equation*}
\end{example}

\subsection{Existence}

A \emph{Frobenius category} is an exact category with enough projectives and with enough injectives and the projectives and injectives coincide. Let $\cat{E}$ be a Frobenius category with $\cat{P}$ the full subcategory of projective-injective objects. The ideal quotient $q\colon \cat{E} \to \stab{\cat{E}}$ with respect to the morphisms factoring through $\cat{P}$ has a natural triangulated structure by \cite[I.2]{Happel:1988}. A triangulated category is \emph{algebraic}, if it is triangle equivalent to $\stab{\cat{E}}$ for some Frobenius category; see \cite[3.6]{Keller:2006b}. 

The key observation for the proof of \cref{main} is the following result, which is stated in \cite[3.2]{Keller/Vossieck:1987}.

\begin{proposition} \label{kcat_stable}
Let $\cat{E}$ be a Frobenius category with $\cat{P}$ the full subcategory of projective-injective objects and $q \colon \cat{E} \to \stab{\cat{E}}$ the canonical functor. Let $\cat{C} \subseteq \stab{\cat{E}}$ be a non-negative full subcategory and set $\cat{B} \colonequals q^{-1}(\cat{C})$. Then the functor $\cat{B} \to \cat{C}$ induces an equivalence of triangulated categories
\begin{equation*}
\kbcat{\cat{B}}/\kbcat{\cat{P}} \to \kbcat{\cat{C}}\,.
\end{equation*}
\end{proposition}

Note, that in the equivalence connects the Verdier quotient of the homotopy category and the homotopy category of an ideal quotient. We postpone the proof to \cref{sec:proof}. 

\begin{remark}
In the \namecref{kcat_stable} we show that the tilting subcategory $\cat{B}$ in $\kbcat{\cat{B}}$ is send to the tilting subcategory $\cat{C}$ under the Verdier quotient functor $\kbcat{\cat{B}} \to \kbcat{\cat{B}}/\kbcat{\cat{P}}$. In general, Verdier quotients need not preserve tilting subcategories. 
\end{remark}

Without the assumption that the subcategory $\cat{C}$ is non-negative the \cref{kcat_stable} is false in general:

\begin{example} \label{counterexample_nonnegative}
Let $k$ be a field and $A = k[x]/(x^2)$. Then $\cat{E} = \mod{A}$ is a Frobenius category and $\stab{\cat{E}} = \mod{k}$ is the category of finite-dimensional vector spaces which is a triangulated category with $\susp = \id$. We show below that $\kbcat{\mod{A}}/\kbcat{\proj{A}}$ is not equivalent to $\kbcat{\mod{k}}$, that is that the conclusion of \cref{kcat_stable} does not hold for $\cat{C} = \stab{\cat{E}}$, which is not non-negative.
Observe first that $\kbcat{\mod{k}}= \dbcat{\mod{k}}$ has no non-trivial thick subcategory. But on the other hand $\kbcat{\mod{A}}/\kbcat{\proj{A}}$ admits a non-trivial Verdier quotient 
\[
\kbcat{\mod{A}}/\kbcat{\proj{A}} \to \dbcat{\mod{A}}/\kbcat{\proj{A}} \cong \stab{\cat{E}}\,;
\]
see \cite[Theorem~4.4.1]{Buchweitz:2021}. In particular, the kernel of this Verdier quotient is a non-trivial thick subcategory. Therefore they can not be triangle equivalent.
\end{example}

\begin{proposition} \label{existence_weak_real}
Let $\cat{E}$ be a Frobenius category with $\cat{P}$ the full subcategory of projective-injective objects and $q \colon \cat{E} \to \stab{\cat{E}}$ the canonical functor. Let $\cat{C} \subseteq \stab{\cat{E}}$ be an admissible exact subcategory. Then there exists a weak realization functor $\kbcat{\cat{C}} \to \stab{\cat{E}}$. 
\end{proposition}
\begin{proof}
Set $\cat{B} \colonequals q^{-1}(\cat{C})$. By \cref{kcat_stable} there exists an equivalence of triangulated categories 
\begin{equation*}
\sF \colon \kbcat{\cat{B}}/\kbcat{\cat{P}} \to \kbcat{\cat{C}}\,.
\end{equation*}
There is also an equivalence
\begin{equation*}
\sB \colon \stab{\cat{E}} \to \dbcat{\cat{E}}/\kbcat{\cat{P}}\,;
\end{equation*}
this has been stated in \cite[Example~2.3]{Keller/Vossieck:1987} with proofs provided in \cite[Corollary~2.2]{Iyama/Yang:2020} or \cite[Proposition~4.4.18]{Krause:2022}. Then the following composition involving the quasi-inverses of the above functors yields the claim
\begin{equation*}
\kbcat{\cat{C}} \xrightarrow{\sF^{-1}} \kbcat{\cat{B}}/\kbcat{\cat{P}} \to \kbcat{\cat{E}}/\kbcat{\cat{P}} \xrightarrow{\sB^{-1}} \stab{\cat{E}}\,. \qedhere
\end{equation*}
\end{proof}

\begin{proof}[Proof of Theorem \ref{main}]
By \cref{existence_weak_real} there exists a weak realization functor, and it induces a realization functor by \cref{weak_real_to_real}.
\end{proof}

From \cref{kcat_stable} we can also deduce the following corollary. 

\begin{corollary} \label{CorOfMain}
Let $\cat{C}$ be an admissible exact subcategory of $\underline{\cat{E}}$. Then $\cat{B}=q^{-1} (\cat{C} )$ is extension-closed in $\cat{E}$ and the functor $q \colon \cat{B} \to \cat{C}$ sends exact sequences to exact triangles. In this case $q$ induces a triangle equivalence 
\[
\dbcat{\cat{B}}/\kbcat{\cat{P}} \to \dbcat{\cat{C}}\,.
\]
\end{corollary}
\begin{proof}
It is straightforward to check that $\kbcat{\cat{P}}$ and $\Ac[b]{\cat{B}}$ are Hom-orthogonal in $\kbcat{\cat{B}}$. Then $\Ac[b]{\cat{B}}$ is a full subcategory of $\kbcat{\cat{B}}/\kbcat{\cat{P}}$ by \cite[Proposition~1.6.10]{Kashiwara/Schapira:1990}. So it is enough to show that the equivalence from \cref{kcat_stable} restricts to an equivalence of the acyclic complexes $\Ac[b]{\cat{B}} \to \Ac[b]{\cat{C}}$. 

The fully faithfullness of the restriction holds as $\Ac[b]{\cat{B}}$ is a full subcategory of $\dbcat{\cat{B}}/\kbcat{\cat{P}}$. Essentially surjectivity holds as 
\begin{equation*}
\Ext_{\cat{B}}^1(X,Y) \cong \Hom_{\stab{\cat{E}}}(X,\susp Y) \cong \Ext^1_{\cat{C}}(X,Y)
\end{equation*}
for any $X,Y \in \cat{B}$. 
\end{proof}

\subsection{Fully faithfulness and equivalence} \label{equivalence}

Let $\cat{C}$ be an admissible exact subcategory of a triangulated category $\cat{T}$. In this section we discuss when a realization functor
\[
\sR \colon \dbcat{\cat{C}} \to \cat{T}
\]
is fully faithful and even an equivalence. The realization functor $\sR$ induces natural group homomorphisms 
\[
\Phi_n(X,Y) \colonequals (\Ext^n_{\cat{C}}(X,Y) \xrightarrow{\cong} \Hom_{\dbcat{\cat{C}}}(X,\susp^n Y) \xrightarrow{\sR} \Hom_\cat{T} (X, \susp^n Y))
\]
for $X,Y \in \cat{C}$ and $n \in \BZ$. Here $\Ext^n_{\cat{C}}$ are the groups of Yoneda extensions for $n \geq 0$ and we set $\Ext^n_{\cat{C}} \colonequals 0$ for $n<0$. For the isomorphism see for example \cite[Proposition~4.2.11]{Krause:2022}. These natural morphisms have been considered in \cite[Lemma~2.11]{Chen/Han/Zhou:2019} for hearts of t-structures and in \cite[A.8]{Positselski:2011} for exact subcategories. The morphism $\Phi_n(X,Y)$ is an isomorphism for $n < 0$ as $\cat{C}$ is non-negative, for $n = 0$ as $\cat{C}$ is full, and for $n = 1$ by \cite[Corollary~A.17]{Positselski:2011}. Further, for $n = 2$ it is a monomorphisms by \cite[Corollary~A.17]{Positselski:2011}. The following result appears in \cite[Remarque~3.1.17]{Beilinson/Bernstein/Deligne:1982} and \cite[Lemma~2.11]{Chen/Han/Zhou:2019} when $\cat{C}$ is the heart of a bounded t-structure. 

\begin{lemma} \label{ffreal}
Let $\cat{C}$ be an admissible exact subcategory of $\cat{T}$ and let $\sR$ be a realization functor of $\cat{C}$. Then the following are equivalent
\begin{enumerate}
\item \label{ffreal:ff} $\sR$ is fully faithful;
\item \label{ffreal:Phi_iso} $\Phi_n(X,Y)$ is an isomorphism for all $X,Y \in \cat{C}$ and $n \in \BZ$;
\item \label{ffreal:Phi_surj} $\Phi_n(X,Y)$ is surjective for all $X,Y \in \cat{C}$ and $n \in \BZ$;
\item \label{ffreal:defl} For every $X,Y \in \cat{C}$, $n \geq 1$ and every morphism $f \colon X\to \susp^n Y$ in $\cat{T}$ there exists a $\cat{C}$-deflation $d\colon Z \to X$ with $f\circ d=0$ in $\cat{T}$; and
\item[(4${}^{\rm op}$)] \label{ffreal:infl} For every $X,Y \in \cat{C}$, $n \geq 1$ and every morphism $f \colon X\to \susp^n Y$ in $\cat{T}$ there exists an $\cat{C}$-inflation $i \colon Y \to W$ such that $\susp^n i\circ f=0$ in $\cat{T}$.
\end{enumerate}
\end{lemma}
\begin{proof}
The implication \cref{ffreal:ff} $\implies$ \cref{ffreal:Phi_iso} is clear and the converse is an application of \emph{d\'evissage} using $\dbcat{\cat{C}} = \thick_{\dbcat{\cat{C}}}(\cat{C})$; see for example \cite[Lemma 3.1.8]{Krause:2022}. 

The implication \cref{ffreal:Phi_iso} $\implies$ \cref{ffreal:Phi_surj} is clear and the converse is shown in \cite[Corollary~A.17]{Positselski:2011}. 

A standard construction shows that \cref{ffreal:Phi_iso} is equivalent to
\begin{enumerate}
\setcounter{enumi}{4}
\item Every $f\colon X \to \susp^n Y$ in $\cat{T}$ with $X,Y \in \cat{C}$ and $n \geq 1$ decomposes as
\begin{equation*}
X = X_0 \to \susp X_1 \to \susp^2 X_2 \to \cdots \to \susp^n X_n = \susp^n Y
\end{equation*}
for $X_i \in \cat{C}$;
\end{enumerate}
see for example \cite[Lemma 2.1]{Chen/Han/Zhou:2019} for the abelian case. Moreover, by induction over $n$ this is also equivalent to
\begin{enumerate}
\setcounter{enumi}{5}
\item \label{ffreal:susp} Every $f\colon X \to \susp^n Y$ in $\cat{T}$ with $X,Y \in \cat{C}$ and $n \geq 1$ decomposes as $X \to \susp U\to \susp^nY$ for some $U \in \cat{C}$.
\end{enumerate}

So it is enough to show that \cref{ffreal:defl} and \cref{ffreal:susp} are equivalent. For the backward direction it is enough to observe that any morphism $X \to \susp U$ in $\cat{T}$ with $X,U \in \cat{C}$ induces an exact sequence $U \to Z \xrightarrow{d} X$ in $\cat{C}$. For the forward direction let $f \colon X \to \susp^n Y$ be a morphism in $\cat{T}$ with $X,Y \in \cat{C}$ and $n \geq 1$. Then there exists a deflation $d \colon Z \to X$ such that $f \circ d = 0$. We complete $d$ to an exact sequence $U \to Z \xrightarrow{d} X$ in $\cat{C}$. Then $f$ factors through the induced morphism $X \to \susp U$. This shows \cref{ffreal:susp}. 

The equivalence of \cref{ffreal:Phi_iso} and \cref{ffreal:infl} holds by an analogous argument.
\end{proof}

\begin{remark}
The previous Lemma can be strengthened to yield an explicit description of the image of $\Phi_n(X,Y)$. That is, the subgroup $\im(\Phi_n(X,Y))$ is the set of all morphisms $f \colon X \to \susp^n Y$ with $X,Y \in \cat{C}$ such that there exists a $\cat{C}$-deflation $d \colon Z \to X$ such that $f\circ d=0$.     
\end{remark}

For a subcategory $\cat{C}$ of a triangulated category $\cat{T}$ we denote by $\thick_\cat{T}(\cat{C})$ the smallest thick subcategory of $\cat{T}$ that contains $\cat{C}$.

\begin{corollary}
Let $\cat{C}$ be an admissible exact subcategory of $\cat{T}$. A realization functor of $\cat{C}$ is an equivalence of triangulated categories if and only if it is fully faithful and $\thick_{\cat{T}}(\cat{C}) = \cat{T}$. \qed
\end{corollary}

\begin{example}
Let $\cat{C}$ be a fully exact subcategory of $\cat{E}$. Then the induced functor $\sF \colon \dbcat{\cat{C}} \to \dbcat{\cat{E}}$ is a realization functor for $\cat{C} \subseteq \dbcat{\cat{E}}$. The functor $\sF$ is fully faithful if and only if the inclusion $\cat{C} \subseteq \cat{E}$ induces isomorphism on the Ext-groups. For example, the latter condition is satisfied by resolving subcategories; see \cite[Section~2]{Auslander/Reiten:1991b} and also \cite[Definition~5.1]{Henrard/Kvamme/vanRoosmalen:2022}. 

The functor $\sF$ is an equivalence if additionally $\cat{E}$ is the smallest additively-closed subcategory closed under the 2-out-of-three property containing $\cat{C}$. For example, this is satisfied by finitely resolving subcategories; cf.\@ \cite[Theorem~3.11(2)]{Henrard/vanRoosmalen:2020c}.
\end{example}

\section{\texorpdfstring{Proof of the main \namecref{kcat_stable}}{Proof of the main Proposition}} \label{sec:proof}

For clarity we use different notations for the suspension in the stable category and the homotopy category. We write $\susp$ for the suspension or shift functor in $\stab{\cat{E}}$ where $\cat{E}$ is a Frobenius exact category. By construction, we have $q(\Omega^n X) = \susp^{-n} X$ for any $X \in \cat{E}$ where $\Omega$ is the syzygy functor. On the other hand, for an additive category $\cat{A}$ we write $\ccat{\cat{A}}$ for the category of chain complexes. In $\ccat{\cat{A}}$ and the homotopy category $\kcat{\cat{A}}$, we denote the degree $n$ shift of a complex $X$ by $X[n]$; this is the complex given by
\begin{equation*}
X[n]^\ell = X^{\ell+n} \quad \text{and} \quad d_{X[n]} = (-1)^n d_X\,.
\end{equation*}
For a \emph{map of complexes} $f \colon X \to Y$ we write
\begin{equation*}
\partial(f) = d^Y f - f[-1] d^X \colon X \to Y[-1]\,.
\end{equation*}
The map $f$ is a \emph{chain map} if and only if $\partial(f) = 0$. Note, that a map of complexes need not commute with the differential, while a chain map does.

\begin{lemma} \label{functor_full}
Let $\cat{E}$ be a Frobenius exact category with $\cat{P}$ the full subcategory of projective-injective objects and $q \colon \cat{E} \to \stab{\cat{E}}$ the canonical functor. Let $\cat{C} \subseteq \stab{\cat{E}}$ be a non-negative full subcategory and set $\cat{B} \colonequals q^{-1}(\cat{C})$. For any chain map $f \colon q(X) \to q(Y)$ in $\ccat{\stab{\cat{E}}}$ with $X \in \clcat{\cat{B}}$ and $Y \in \crcat{\cat{B}}$ there exist chain maps $g \colon \hat{X} \to Y$ and $s \colon \hat{X} \to X$ with $\hat{X} \in \clcat{\cat{B}}$ and $\cone(s) \in \cbcat{\cat{P}}$ such that $q(g) = f \circ q(s)$. 
\end{lemma}
\begin{proof}
First we construct an injective resolution $I$ of $X$ in the category of complexes. By \cite[4.1, Lemma, b)]{Keller:1990}, there exists a left bounded complex $I_0$ of projective-injective objects and a chain map $j_0 \colon X \to I_0$ that is an inflation in each degree. We denote the cokernel of $j_0$ by $q_0 \colon I_0 \to \Omega^{-1} X$. Continuing this process, we obtain a sequence of chain maps
\begin{equation*}
\begin{tikzcd}[column sep=small,row sep=small]
&[+2em] & \Omega^{-1} X \ar[dr,tail,"j_1"] && \Omega^{-2} X \ar[dr,tail,"j_2"] && \Omega^{-3} X \ar[dr,tail] \\
X \ar[r,tail,"h_{-1} = j_0"] & I_0 \ar[rr,"h_0" swap] \ar[ur,two heads,"q_0"] && I_1 \ar[rr,"h_1" swap] \ar[ur,two heads,"q_1"] && I_2 \ar[rr] \ar[ur,two heads,"q_2"] && \cdots \nospacepunct{.}
\end{tikzcd}
\end{equation*}
We set $h_{-1} \colonequals j_0$ and $h_\ell \colonequals j_{\ell+1} q_\ell$. As $X$ is left bounded we may assume that there exists an integer $s$ such that $(I_\ell)^{\leq s} = 0$ for all $\ell$; that is $s$ is a universal lower bound. Since the maps $j_\ell$ are degreewise inflations, every map from $\Omega^{-\ell} X$ to a complex of projective-injective objects factors through $j_\ell$.

We take a lift of $f$ to a map of complexes $\hat{f} \colon X \to Y$ in $\ccat{\cat{E}}$. This map need not commute with the differential. However, as it is the lift of a chain map in $\stab{\cat{E}}$ the map $\partial(\hat{f})$ factors through a complex of projective-injective objects. So there exists a map $g_0 \colon I_0 \to Y[-1]$ such that $\partial(\hat{f}) = g_0 j_0$. For convenience we set $q_{-1} \colonequals \id_X$ and $g_{-1} \colonequals \hat{f}$. We now inductively construct maps $g_\ell \colon I_\ell \to Y[-\ell-1]$ with $\partial(g_{\ell-1}) = g_\ell j_\ell q_{\ell-1}$. 

We assume that we have constructed the maps for any integer $\leq \ell$ for some $\ell \geq 0$. Then $0 = \partial(g_\ell) j_\ell q_{\ell-1}$ as $j_\ell$ and $q_{\ell-1}$ are chain maps. As $q_{\ell-1}$ consists of deflations in each degree, we get $0 = \partial(g_\ell) j_\ell$. Hence $\partial(g_\ell)$ factors through $q_\ell$ and we obtain the commutative diagram
\begin{equation*}
\begin{tikzcd}
\Omega^{-\ell} X \ar[r,tail,"j_\ell"] \ar[dr,"0" swap] & I_\ell \ar[r,two heads,"q_\ell"] \ar[d,"\partial(g_\ell)" description] & \Omega^{-\ell-1} X \ar[dl] \ar[d,tail,"j_{\ell+1}"] \\
& Y[-\ell-2] & I_{\ell+1} \ar[l,dashed,"g_{\ell+1}"] \nospacepunct{.}
\end{tikzcd}
\end{equation*}
By the non-negativity of $\cat{C}$, we have 
\begin{equation*}
\Hom_{\ccat{\stab{\cat{E}}}}(q(\Omega^{-\ell-1} X),q(Y[-\ell-2])) = \Hom_{\ccat{\stab{\cat{E}}}}(\susp^{\ell+1} q(X),q(Y[-\ell-2])) = 0\,.
\end{equation*}
Hence the map $\Omega^{-\ell-1} X \to Y[-\ell-2]$ factors through $j_{\ell+1}$. Note, that $g_{\ell+1}$ need not be a chain map. We continue this process until the map $g_{\ell+1}$ is a chain map. As $Y$ is right bounded and the $I_\ell$'s have a universal upper bound, this will happen eventually.

Let $t$ be an integer such that $Y^{\geq t} = 0$. We replace $I_\ell$ by the truncation $(I_\ell)^{\geqslant t-\ell-1}$. This does not effect the properties of the $g_\ell$'s, as they are zero in the other degrees. To summarize, we have a sequence of maps
\begin{equation*}
\begin{tikzcd}
X \ar[r,tail,"h_{-1}"] \ar[d,"\hat{f} = g_{-1}"] & I_0 \ar[r,"h_0"] \ar[d,"g_0"] & I_1 \ar[r] \ar[d,"g_1"] & \cdots \ar[r] & I_{n-1} \ar[r,"h_{n-1}"] \ar[d,"g_{n-1}"] & I_n \ar[d,"g_n"] \\
Y & Y[-1] & Y[-2] & \cdots & Y[-n] & Y[-n-1] \nospacepunct{,}
\end{tikzcd}
\end{equation*}
where each $I_\ell$ is a bounded complex of projective-injective objects, $g_n$ is a chain map and $\partial(g_{\ell-1}) = g_\ell h_{\ell-1}$ and $h_\ell h_{\ell-1} = 0$ for $0 \leq \ell \leq n$. 

We take the total complex $J$ of $I_0 \to \cdots \to I_n$. This means as graded module $J = \bigoplus I_i[i]$ with differential
\begin{equation*}
\left. d_J \right|_{I_i[i]} = d_{I_i[i]} + (-1)^i h_i[i]\,.
\end{equation*}
For convenience we use a nonstandard sign convention. We set
\begin{equation*}
v \colonequals \sum_i g_i[i] \colon J \to Y[-1]\,.
\end{equation*}
This is a chain map, as
\begin{equation*}
\begin{aligned}
\left. (v d_J) \right|_{I_i[i]} &= g_i[i-1] d_{I_i[1]} + (-1)^i g_{i+1}[i] h_i[i] \\
&= (-1)^i (g_i[-1] d_{I_i} + g_{i+1} h_i)[i] \\
&= (-1)^i (d_{Y[-i-1]} g_i)[-i] = d_{Y[-1]} g_i[i] = \left. (d_{Y[-1]} v) \right|_{I_i[i]}\,.
\end{aligned}
\end{equation*}
One can similarly check that the composition $u \colonequals (X \to I_0 \to J)$ is a chain map. By construction we have $\partial(\hat{f}) = vu$. Then $\hat{X} = \susp^{-1} \cone(u)$ and $g = (-v, \hat{f})$ and the natural map $s \colon \hat{X} \to X$ satisfy the desired properties. 
\end{proof}

\begin{lemma} \label{functor_sincere}
Let $\cat{E}$ be a Frobenius exact category with $\cat{P}$ the full subcategory of projective-injective objects and $q \colon \cat{E} \to \stab{\cat{E}}$ the canonical functor. Let $X \in \kbcat{\cat{E}}$. If $q(X) = 0$ in $\kbcat{\stab{\cat{E}}}$, then $X \in \kbcat{\cat{P}}$. 
\end{lemma}
\begin{proof}
It is enough to show the claim for a complex of the form
\[
X = (\cdots \to 0 \to X^{0} \xrightarrow{d^0} X^1 \xrightarrow{d^1} X^2 \to \cdots \to X^{n-1} \xrightarrow{d^{n-1}} X^n \to 0 \to \cdots)
\]
for any $n \geq 0$. We use induction on $n$. 

For $n = 0$, the assumption $q(X) = 0 $ implies $q(X^0) = 0$. Hence $X^0 \in \cat{P}$.

Let $n \geq 1$. As $q(X) = 0$, the morphism $q(d^0)$ is a split monomorphism in $\stab{\cat{E}}$ and there exists a morphism $s\colon X^1\to X^0$ such that $s d^0 - \id_{X^0} = ba$ for some morphisms $X^0 \xrightarrow{a} P \xrightarrow{b} X^0$ with $P \in \cat{P}$. We view $P$ as a complex concentrated in degree zero and set
\begin{equation*}
X' \colonequals \cone(\susp^{-1} a) = (\cdots \to 0 \to X^0 \xrightarrow{\left(\begin{smallmatrix} d^0 \\a \end{smallmatrix}\right)} X^1\oplus P \xrightarrow{\left(\begin{smallmatrix}d^1 & 0\end{smallmatrix}\right)} X^2 \xrightarrow{d^2} X^3 \to \cdots)\,.
\end{equation*}
Since $\left(\begin{smallmatrix} s & -b \end{smallmatrix}\right) \circ \left(\begin{smallmatrix} d^0 \\a \end{smallmatrix}\right) = \id_{X^0}$, the zero differential of $X'$ is a split monomorphism in $\cat{E}$. Therefore, in $\kbcat{\cat{E}}$, the complex $X'$ is isomorphic to a complex $Y$ concentrated between degrees 1 and $n$. In $\kbcat{\stab{\cat{E}}}$ we have $q(Y) \cong q(X') \cong q(X) = 0$. By induction hypothesis we have $X' \cong Y \in \kbcat{\cat{P}}$. By construction there is an exact triangle $X' \to P \to X \to \susp X'$, and as $X', P \in \kbcat{\cat{P}}$, so is $X$.
\end{proof}

\begin{proof}[Proof of \cref{kcat_stable}]
As $q(\kbcat{\cat{P}}) = 0$ in $\kbcat{\stab{\cat{E}}}$, the functor $q \colon \kbcat{\cat{B}} \to \kbcat{\cat{C}}$ induces a triangle functor
\begin{equation*}
\overline{q} \colon \kbcat{\cat{B}}/\kbcat{\cat{P}} \to \kbcat{\cat{C}}\,.
\end{equation*}
We claim that $\overline{q}$ is an equivalence of triangulated categories. For this we need to show that $\overline{q}$ is full, faithful and essentially surjective. 

The functor $\overline{q}$ is full by \cref{functor_full}. By \cref{functor_sincere}, whenever $\overline{q}(X) = 0$ then $X = 0$. As we already know that $\overline{q}$ is full, this implies that $\overline{q}$ is faithful by \cite[p.~446]{Rickard:1989}; also see \cite[4.3, 4.4]{Pu/Ringel:2015}. 

It remains to show that $\overline{q}$ is essentially surjective. The essential image of $\overline{q}$ is a thick subcategory containing the complexes concentrated in degree zero. As these complexes generate $\kbcat{\cat{C}}$, the functor $\overline{q}$ is essentially surjective.
\end{proof}

\bibliographystyle{amsalpha}
\bibliography{Lit}

\end{document}